\documentclass[reqno,11pt]{amsart}

\usepackage[T1]{fontenc} 

\usepackage{amssymb}
\usepackage{amsmath}
\usepackage{amsfonts}
\usepackage{amsthm}
\usepackage{mathtools}
\usepackage{mathabx}
\usepackage{dsfont}

\usepackage{enumerate}

\usepackage{fullpage}
\usepackage{soul}

\usepackage{hyperref}

\theoremstyle{plain}%
\newtheorem{theorem}{Theorem}[section]
\newtheorem{lemma}[theorem]{Lemma}
\newtheorem{proposition}[theorem]{Proposition}

\newtheorem{conjecture}[theorem]{Conjecture}
\newtheorem*{conjecture*}{Conjecture} 

\theoremstyle{definition}
\newtheorem{definition}[theorem]{Definition}

\theoremstyle{remark}
\newtheorem{remark}[theorem]{Remark} 

\DeclareMathOperator{\Var}{Var} 
\DeclareMathOperator{\diag}{diag}

\usepackage{scalerel,stackengine}
\newcommand\pig[1]{\scalerel*[5.5pt]{\Big#1}{%
  \ensurestackMath{\addstackgap[1.5pt]{\big#1}}}}
\newcommand\pigl[1]{\mathopen{\pig{#1}}}
\newcommand\pigr[1]{\mathclose{\pig{#1}}}

\title{When are sequences of Boolean functions tame?} 

\author{Malin P. Forsstr\"om}
\address[Malin P. Forsstr\"om]{Department of Mathematics, KTH Royal Institute of Technology, 100 44 Stockholm, Sweden.}
\email{malinpf@kth.se} 
 
\keywords{Boolean functions}
\subjclass[2010]{60K99}

\begin{document}

\begin{abstract}
    In~\cite{js2006}, Jonasson and Steif conjectured that no non-degenerate sequence of transitive  Boolean functions \( (f_n)_{n \geq 1}\) with \( \lim_{n \to \infty} I(f_n)= \infty \) could be tame (with respect to some \( (p_n)_{n \geq 1} \)). In a companion paper~\cite{f}, the author showed that this conjecture in its full generality is false, by providing a counter-example for the case when, at the same time,  \( \lim_{n \to \infty} np_n = \infty \) and \( \lim_{n \to \infty} n^\alpha p_n = 0\) for some \( \alpha \in (0,1 ) \). In this paper we show that with slightly different assumptions, the conclusion of the conjecture holds when the sequence \( (p_n)_{n \geq 1}\) is bounded away from zero and one.
\end{abstract}

 \maketitle

\section{Introduction}

The aim of this paper is to show that with slightly different assumptions, a conjecture made by Jonasson and Steif in~\cite{js2006} is true. The results included in this paper complements the earlier paper~\cite{f} by the author, which shows that the conjecture, in its full generality, is false. 

Before presenting the conjecture and the main result of this paper, we will need some notation and definitions, which we now introduce. 
For each \( n \geq 1 \), fix some \( p_n \in (0,1) \) and let \( {X^{(n)} = (X_t^{(n)})_{t \geq 0}}\) be the continuous time \( p_n \)-biased random walk on the \( n \)-dimensional hypercube \( \{ 0,1 \}^n \) defined as follows.  For each \( i \in [n] \coloneqq \{1,2, \ldots, n \} \), let \( (X_t^{(n)}(i))_{t \geq 0} \) be the continuous time Markov chain on \( \{ 0,1 \} \) which at random times, distributed according to a rate one Poisson process, is assigned a new value, chosen according to   \(  (1-p_n) \delta_0  + p_n \delta   \), independently of everything else. For each \( t \geq 0 \), we let \( X_t^{(n)} \coloneqq \bigl( X_t^{(n)}(1), \ldots, X_t^{(n)}(n)\bigr)\). The unique stationary distribution of $(X^{(n)}_t)_{t\geq  0}$, denoted by  \( \pi_{n}\), is the   measure \( ((1-p_n) \delta_0  + p_n \delta_1 )^{\otimes n} \) on \( \{0,1 \}^n \). Throughout this paper, we will always assume that  \( X_0^{(n)}   \) is chosen with respect to this measure. When  \( t > 0 \) is small, the difference between \( X_0^{(n)} \) and \( X_t^{(n)} \) is often thought of as noise which describes a small proportion of the bits having been miscounted or corrupted.

A function \( f \colon \{ 0,1 \}^{n} \mapsto \{0,1\} \) will be referred to as a \emph{Boolean function}.  Some classical examples of Boolean functions are the \emph{Dictator function} \( \textsc{Dict}_n(x) \coloneqq x(1) \), the \emph{Majority function} \( {\textsc{Maj}_n(x) \coloneqq \mathds{1} \bigl(\sum_{i=1}^{n} x(i) \geq n/2  \bigr)} \) and the \emph{Parity function} \( {\textsc{Parity}_n(x) \coloneqq \mathds{1} \bigl( \sum_{i=1}^{n} x(i) \text{ is even} \bigr)}\)  (see, e.g.,~\cite{gs, odonnell}). 
In this paper we will in general be interested in sequences \( (f_n)_{n \in \mathcal{N}} \) of Boolean functions, with \( f_n \colon \{ 0,1 \}^n \to \{ 0,1 \} \).
Since it is sometimes not natural to require that a sequence of Boolean functions is defined for each \( n \in \mathbb{N} \), we only require that a sequence of Boolean functions is defined for \( n \) in an infinite increasing sub-sequence \( \mathcal{N} \) of \( \mathbb{N} \). Such sub-sequences of \( \mathbb{N} \) will be denoted by \( \mathcal{N} = \{ n_1,n_2, \ldots \} \), where \( 1 \leq n_1 < n_2 < \ldots \). To simplify notation, whenever we consider the limit of a sequence \( (y_{n_m})_{m \geq 1} \) and the dependency on \( \mathcal{N} \) is clear, we will abuse notation and write \( \lim_{n \to \infty} y_n \) instead of \( \lim_{m\to \infty} y_{n_m}\).  
 
One of the main objectives of~\cite{js2006} was to, given a sequence of Boolean functions \( (f_n)_{n \in \mathcal{N}}\), introduce notation which describes possible behaviours of \( \bigl(f_n(X_t^{(n)})\bigr)_{t \geq 0} \) for large \( n \). Two of these definitions, which will be relevant for the current paper, are given in the following definition.

\begin{definition}
    Let \( (f_n)_{n \in \mathcal{N}}\), \( f_n \colon \{ 0,1 \}^{n} \to \{ 0,1 \} \), be a sequence of Boolean functions. For \( n \in \mathcal{N} \), let \( C_{f_n}   \) denote the (random) number of times in \( (0,1) \) at which \( (f_n(X_t^{(n)}))_{t \geq 0} \) has changed its value, i.e., let  
    \begin{equation*}
        C_{f_n} \coloneqq  \lim_{N \to \infty}  \sum_{i=0}^{N-1} \mathds{1} \bigl({f_n(X^{(n)}_{i/N}) \neq f_n(X^{(n)}_{(i+1)/N})}\bigr).
    \end{equation*} 
    Then \( (f_n)_{n \in \mathcal{N}} \) is said to be
    \begin{enumerate}[(i)] 
	    \item \emph{tame} if   \( (C_{f_n})_{n\geq 1} \) is tight, that is for every \( \varepsilon > 0 \) there is \( k \geq 1 \) and \( n_0 \geq 1 \) such that 
	    \begin{equation*}
		    P(C_{f_n}\geq k) < \varepsilon \qquad \forall n \in \mathcal{N} \colon n \geq n_0
	    \end{equation*}
	    \item \emph{volatile} if \( C_{f_n} \Rightarrow \infty \) in distribution.
    \end{enumerate} 
\end{definition}

In~\cite{js2006}, the authors showed that a sequence of Dictator functions is tame and that a sequence of Parity functions is volatile, while a sequence of Majority functions is neither tame nor volatile. More generally, the authors proved that any noise sensitive sequence of Boolean functions  is volatile, while any sequence of Boolean functions which is tame is noise stable. As noted in~\cite{f}~and~\cite{pg}, there are many sequences of functions which are both noise stable and volatile, and hence the opposite does not hold.

We now describe a few additional properties which a Boolean function can have.
First, a Boolean function \( f \colon \{ 0,1 \}^{n} \to \{ 0,1 \} \) is said to be \emph{transitive} if for all \( i,j \in [n] \coloneqq \{ 1,2, \ldots, n \} \) there is a permutation \( \sigma \) of \( [n] \) which is such that (i) \( \sigma(i) = j \), and (ii) for all \( x \in \{ 0,1 \}^{n} \), if we define \( {\sigma(x) \coloneqq \bigl(x(\sigma(k))\bigr)_{k \in [n]}} \), then \( f(x) = f\bigl(\sigma(x)\bigr) \). We say that \( f \colon \{ 0,1 \}^n \to \{ 0,1 \} \) is \emph{increasing} if for all \( x,x' \in \{ 0,1 \}^n \) such that \( x(i) \leq x'(i) \) for all \(i \in [n] \), we have  \( f(x) \leq f(x') \). A sequence of Boolean functions \( (f_n)_{n \in \mathcal{N}} \) is said to be \emph{non-degenerate} if 
\begin{equation*}
    0 < \liminf_{n \to \infty } P\bigl(f_n(X_0^{(n)}) = 1\bigr) \leq \limsup_{n \to \infty } P\bigl(f_n(X_0^{(n)}) = 1\bigr)  < 1 .
\end{equation*} 

Given \( x \in \{ 0,1 \}^n \) and \( i \in [n] \), let \( R_i x \) denote the random element in \( \{ 0,1 \}^n \) obtained by resampling the \( i\)th bit of \( x \) according to \( (1-p_n)\delta_0 + p_n \delta_1 \). If \( f \colon \{ 0,1 \}^n \to \{ 0,1 \} \), we define the \emph{influence} of the \( i \)th bit on \( f \), by 
\begin{equation*}
    I_i(f) \coloneqq P\bigl(f(X_0^{(n)}) \neq f(R_i X_0^{(n)}) \bigr).
\end{equation*}
Note that this definition differs from the definition of influences in, e.g.,~\cite{gk2018},~\cite{lm2019}, and~\cite{odonnell}  by a factor \( 2p_n(1-p_n) \), but agrees with the analogue definitions given in, e.g.,~\cite{f2020} and~\cite{js2006}.
 If \( I_i(f) \) is the same for all \(i \in [n] \coloneqq \{ 1,2, \ldots, n \}\), then we say that \( f \) is \emph{regular}. The sum of the influences, \( I(f) \coloneqq \sum_{i \in [n]} I_i(f_n) \), is called the \emph{total influence} of \( f\).
 
In~\cite{js2006}, the authors show that a sufficient, but not necessary, condition for a non-degenerate sequence \( (f_n)_{n \in \mathcal{N}} \) of Boolean functions to be tame is that \( \sup_n \mathbb{E}[C_{f_n}]< \infty \) (or equivalently, that \( \sup_n I(f_n)< \infty \)). It is natural to ask if this condition is also necessary for some natural subset of the set of all sequences of Boolean functions. This motivated the following conjecture.
\begin{conjecture}[Conjecture~1.21 in~\cite{js2006}] \label{conjecture: the conjecture}
    If \( (f_n)_{n \in \mathcal{N}} \) is a non-degenerate sequence of transitive Boolean functions and \( \lim_{n \to \infty} \mathbb{E} [C_{f_n}] = \infty \), then \( (f_n)_{n \in \mathcal{N}} \) is not tame.
\end{conjecture}

\begin{remark}
    By the second remark after Conjecture~1.21 in~\cite{js2006}, there is a tame sequence of Boolean functions which satisfies all assumptions except non-degeneracy in Conjecture~\ref{conjecture: the conjecture}, and hence this assumption is necessary. 
\end{remark}

In~\cite{f2020}, a family of counter-examples to Conjecture~\ref{conjecture: the conjecture} was given. In detail, these examples provide counter-examples to Conjecture~\ref{conjecture: the conjecture} exactly when the sequence \( (p_n)_{n \geq 1} \) is such that \( {\lim_{n \to \infty} np_n = \infty} \) and \( \lim_{n \to \infty } np_n^r = 0 \) for some \( r \geq 2 \).  The assumption that \( \lim_{n \to \infty} np_n = \infty \) guarantees that, as \( n \to \infty ,\) the expected number of jumps made by  \( \smash{(X^{(n)})_{n \in \mathcal{N}}} \) in \( (0,1 ) \) tends to infinity. In particular, if \( \limsup_{n \to \infty} np_n < \infty \), then all sequences of Boolean functions will be tame. Consequently, the family of examples given in~\cite{f2020} show that Conjecture~\ref{conjecture: the conjecture} is false whenever the sequence \( (p_n)_{n \geq 1} \) tends to zero sufficiently fast, but still slowly enough for there to be interesting behaviour. 
In contrast to this result, the main objective of this paper is to show that when \( (p_n)_{n \geq 1} \) is bounded away from zero and one, any non-degenerate sequence of increasing and regular  Boolean functions is non-tame, and hence for \( (p_n)_{n\geq 1} \) in this range, a version of Conjecture~\ref{conjecture: the conjecture} holds.

\begin{theorem}\label{theorem: increasing gen p}
	If \( (p_n)_{n \geq 1} \) satisfies \( 0 < \liminf_{n \to \infty} p_n \leq \limsup_{n \to \infty} p_n < 1 \), and  \( (f_n)_{n \in \mathcal{N}} \) is a non-degenerate sequence of regular and increasing Boolean functions, then \((f_n)_{n \in \mathcal{N}} \) is not tame.
\end{theorem}

\begin{remark}
	If we compare the assumptions on \( (f_n)_{n \in \mathcal{N}} \) in Theorem~\ref{theorem: increasing gen p} with the assumptions on  \( (f_n)_{n \in \mathcal{N}} \) in Conjecture~\ref{conjecture: the conjecture}, the property of being increasing is added, however, the requirement of transitivity is replaced with the assumption that \( f_n \) regular for each \( n \in \mathcal{N} \).
\end{remark}

\begin{remark}
	Very interestingly, the family of counter-examples to Conjecture~\ref{conjecture: the conjecture} given in~\cite{f2020} show that the conclusion of Theorem~\ref{theorem: increasing gen p} does not hold if the assumption that \( 0 < \liminf_{n \to \infty} p_n \leq \limsup_{n \to \infty} < 1 \) is replaced by the assumptions that \( \liminf_{n \to \infty} np_n = \infty \) and  \( \limsup_{n \to \infty}n^\alpha p_n < \infty \) for some \( \alpha \in (0,1) \). This behaviour mirrors similar discrepancies between the two regimes for \( (p_n)_{n\geq 1} \) which are present also for other results about Boolean functions (see, e.g.,~\cite{lm2019}). One reason to expect a difference in behaviour between these regimes is that the latter is exactly the regime for which there are non-degenerate sequences of transitive Boolean functions with finite sized witnesses.
\end{remark}

\begin{remark}\label{remark: influence and expected number of jumps}
    Using, e.g., Theorem~1 in~\cite{bkkkl}, one shows that any non-degenerate sequence \( (f_n)_{n \in \mathcal{N}} \) of regular Boolean functions  satisfies \( \lim_{n \to \infty} I(f_n) =\infty\), where \( I(f_n) \) is the so-called total influence of \( f_n \). By Proposition~1.19 in~\cite{js2006}, we have \( \mathbb{E}[C_f] = I(f) \), and hence the assumptions of Theorem~\ref{theorem: increasing gen p} guarantee that \( \lim_{n \to \infty}\mathbb{E}[C_{f_n}] = \infty \). We mention that by definition, if \( (f_n)_{n \in \mathcal{N}} \) is not tame, then this must hold. 
\end{remark}

\begin{remark}
    The proof of Theorem~\ref{theorem: increasing gen p} does not really require that \( f_n \) is regular for each \( n \in \mathcal{N} \), but rather that some positive proportion of the influences are of the same order and correspond to a positive proportion to the total influence. 
\end{remark}

The rest of this paper will be organised as follows. In Section~\ref{sec: background}, we introduce the Fourier-Walsh expansions of Boolean functions, which will be a crucial tool in the proof of Theorem~\ref{theorem: increasing gen p}. In Section~\ref{sec: fourier expression}, we present an expression for \( \mathbb{E}\bigl[C_f^2 \bigr] \) in terms of the Fourier-Walsh coefficients of \( f \). Finally, in Section~\ref{sec: proof of main result}, we give a proof of Theorem~\ref{theorem: increasing gen p}.

\section{Background and notations}\label{sec: background}

In this section, we will give a brief introduction to the Fourier-Walsh expansion of Boolean functions, and state and prove some results which will be useful to us. For a more thorough introduction to the of Fourier analysis to understand  properties of Boolean functions, we refer the reader to~\cite{odonnell}.

For the rest of this section, fix some \( n \geq 1 \) and assume that \( p_n \in (0,1 ) \) is given. To simplify notation, we let \( [n] \) denote the set \( \{ 1,2, \ldots, n \} \).

For functions \( f,g \colon \{ 0,1 \}^n \to \{ 0,1 \} \) and \( X_0^{(n)} \sim \pi_{n} \), we let 
\begin{equation*}
    \langle f,g \rangle \coloneqq \mathbb{E} \bigl[ f(X_0^{(n)}) g(X_0^{(n)}) \bigr].
\end{equation*}
Then \( \langle \cdot , \cdot \rangle \) is an inner product on the set of real-valued functions with domain \( \{ 0,1 \}^n  \).
For \( S \subseteq [n] \) and \( x \in \{ 0,1 \}^n \), define 
\begin{equation*}
     \chi_S(x) \coloneqq \prod_{i \in S}  \frac{x(i) - p_n}{\sqrt{p_n(1-p_n)}} .
\end{equation*}
Then \( \{ \chi_S \}_{S \subseteq [n]} \) is an orthonormal basis for the space of functions \( f \colon \{ 0,1 \}^n \to \mathbb{R} \), using the inner product \( \langle \cdot, \cdot \rangle  \) (see, e.g., Chapter 8.4 in~\cite{odonnell}). In other words, for any \( S,T \subseteq [n] \) we have
\begin{equation}\label{eq: orthonormality}
    \langle \chi_S, \chi_T \rangle = \mathds{1}(S = T).
\end{equation}
Here \( \mathds{1} \) is the indicator function, so that, e.g., \( \mathds{1}(S = T) \) is equal to \( 1 \) if \( S = T \) and equal to \( 0 \) else. Since \( \{ \chi_S \}_{S \subseteq [n]} \) is finite, any function \( f \colon \{ 0,1 \}^n \to \mathbb{R} \) has a unique decomposition
\begin{equation*}
    f(x) = \sum_{S \subseteq [n]} \hat f(S) \chi_S(x),\quad x \in \{ 0,1 \}^{n},
\end{equation*}
where \( \hat f(S) \) is given by \( \langle f , \chi_S \rangle.\) 
Moreover, noting that  for all \( x \in \{ 0,1 \}^n \) and all \( S,T \subset [n] \) we have
\begin{equation}\label{eq: prod of two basis vectors}
     \chi_S(x) \chi_T(x)  = \chi_{S \Delta T}(x)  \prod_{i \in S \cap T} \Bigl( 1 +  \frac{1-2p_n}{\sqrt{p_n(1-p_n)}} \cdot \chi_{\{ i \}}(x)\Bigr) 
\end{equation}
it follows that for any \( S,T,R \subseteq [n] \), 
\begin{equation}\label{eq: tri product expectation}
\begin{split}
    &\mathbb{E} \Bigl[\chi_S \bigl(X_0^{(n)} \bigr) \chi_T \bigl(X_0^{(n)} \bigr) \chi_R \bigl(X_0^{(n)}\bigr) \Bigr] 
    = \mathds{1} \bigl(S \Delta T \Delta R = S \cap T \cap R\bigr) \cdot \Bigl( \frac{1-2p_n}{\sqrt{p_n(1-p_n)}} \Bigr)^{\mathrlap{|S \cap T \cap R|}} . \phantom{{}^{|S \cap T \cap R|}}
\end{split}
\end{equation}  
To simplify the rest of the paper, we will abuse notation slightly and sometimes treat \( \pi_n \) and the functions in \( \{ \chi_S\}_{S \subseteq [n]} \)  as functions with domain \( \{ 0,1 \}^n \), and sometimes as real-valued vectors in \( \mathbb{R}^{\{ 0,1 \}^n} \) in the natural way. We let \( \mathbf{1} \coloneqq (1,1,\dots,1) \in  \mathbb{R}^{\{ 0,1 \}^n}  \), and note that \( \mathbf{1} = \chi_\emptyset \). Analogously, we let \( \mathbf{0} \coloneqq (0,0,\dots,0) \in  \mathbb{R}^{\{ 0,1 \}^n}  \).

For \( i \in [n]\), \( x \in \{ 0,1 \}^n \) and \( y \in \{ 0,1 \} \), let \( x^{i \mapsto y} \in \{ 0,1 \}^n \) be defined by
\begin{equation*}
    x^{i \mapsto y}(j) \coloneqq \begin{cases}
    y &\text{if } j=i \cr
    x(j)   &\text{if } j \neq i,
    \end{cases}
    \qquad j\in [n].
\end{equation*}
Using this notation, for each \( i \in [n] \) we define the differential operator \( D_i \) acting on functions \( f \colon \{ 0,1 \}^n \to \{ 0,1 \} \), by
 \begin{equation*}
     D_i f(x) \coloneqq   f(x^{i \mapsto 1}) - f(x^{i \mapsto 0}),\quad   x \in \{ 0,1 \}^n.
 \end{equation*}
In the next lemma, we use the Fourier-Walsh expansion to describe how these differential operators act on Boolean functions.
 \begin{lemma}\label{lemma: Partial i f}
For any function \( f \colon \{ 0,1 \}^n \to \{ 0,1 \} \), \( i \in [n] \) and \( x \in \{0,1 \}^n \), we have
\begin{equation*}
    D_if(x) 
     = \frac{1}{\sqrt{p_n(1-p_n)}} \sum_{T \subseteq [n] \colon i \not \in T} \hat f \bigl(T \cup \{ i \} \bigr) \chi_T(x).
\end{equation*}
\end{lemma}

For a proof of Lemma~\ref{lemma: Partial i f}, see, e.g., Section~8.4 in~\cite{odonnell}.

Before closing this section, we mention that it is easy to verify that (see, e.g., \cite[Proposition~8.16]{odonnell}),
\begin{equation*}
	\Var f_n = \sum_{S \subseteq [n] \colon S \neq \emptyset} \hat f_n(S)^2,
\end{equation*} 
and similarly (see, e.g., \cite[Proposition 8.23]{odonnell}), that
\begin{equation*}
	I(f_n) = \sum_{S \subseteq [n] \colon S \neq \emptyset} |S|\hat f_n(S)^2.
\end{equation*}

\section{An expression for  \texorpdfstring{\( \mathbb{E}[C_f^2] \)}{the second moment of Cf} using the Fourier coefficients} \label{sec: fourier expression}

The main goal of this section is to give a proof of the following proposition, which will be crucial in the proof of Theorem~\ref{theorem: increasing gen p}.
\begin{proposition}\label{proposition: increasing nonsymmetric}
    Let \( f \colon \{ 0,1 \}^n \to \{ 0,1 \}\) be increasing. Then
    \begin{align*}
        \mathbb{E} [C_f^2]    
        &=   
        \mathbb{E}[C_f] + \mathbb{E}[C_f]^2 
        \\&\qquad +
        \sum_{\substack{S \subseteq [n] \colon \\ S \neq \emptyset }} \frac{e^{-|S|}-(1-|S|)}{|S|^2}        \Bigl[  (1-2p_n)  |S| \hat f(S )  
        + 2\sqrt{p_n(1-p_n)} \sum_{i \in [n] \colon i \not \in S} \hat f \bigl(S \cup \{ i \} \bigr)  
        \Bigr]^2 .
    \end{align*}
\end{proposition}

In order to give a proof of this result, we will first introduce some additional notation. After this, we state and prove a number of lemmas from which the claim of Proposition~\ref{proposition: increasing nonsymmetric} will follow.  

For the rest of this section, assume that \( n \geq 1 \) and \( p_n \in (0,1) \) is given. 
For \( i \in [n]\) and \( x \in \{ 0,1 \}^n \), define
\begin{equation*}
    (x \oplus e_i)(j) \coloneqq \begin{cases}
    1-x &\text{if } j=i \cr
    x(j)   &\text{if } j \neq i,
    \end{cases}
    \qquad j\in [n].
\end{equation*}
We now define two matrices which will be useful throughout the rest of this section. Let \(  Q_n \) be the transition matrix of the discrete time Markov chain indexed by the resampling times of \( (X_t^{(n)})_{t \geq 0} \), i.e.,  for \( x,y \in \{ 0,1 \}^n \) let
  \begin{equation*}
  Q_n(x,y) \coloneqq
  \begin{cases}
       \frac{1}{n} \cdot \sum_{i=1}^n \bigl( (1-p_n) \mathds{1}(x(i)=0) + p_n \mathds{1}(x(i)=1) \bigr) &\text{if } x= y,\cr
       \frac{1}{n} \cdot \bigl( p_n \mathds{1}(x(i)=0) + (1-p_n) \mathds{1}(x(i)=1) \bigr) & \text{if }  y = x \oplus e_i \text{ for some } i \in [n],  \cr
        0 & \text{else.}
    \end{cases}
 \end{equation*}
 Since for any \( x \in \{ 0,1 \}^n \) and \( i \in [n]\) we have \( \mathds{1}(x(i)=1) = x(i) \) and \( \mathds{1}(x(i)=0) = 1-x(i) \), the matrix \( Q_n \) can equivalently be defined by
   \begin{equation*}
   Q_n(x,y) \coloneqq
  \begin{cases}
       \frac{1}{n} \cdot \sum_{i=1}^n \bigl( (1-p_n) (1-x(i)) + p_n x(i) \bigr) &\text{if } x= y, \cr
       \frac{1}{n} \cdot \bigl( p_n (1-x(i))  + (1-p_n) x(i) \bigr)  & \text{if }  y = x \oplus e_i \text{ for some } i \in [n],  \cr
      0 & \text{else.} 
    \end{cases}
 \end{equation*}
One can easily show  that the functions \( \{ \chi_S\}_{S \subseteq [n]} \) are eigenvectors of \( P_n \), and that if we let \( I_n \) denote the \( |\{ 0,1 \}^n| \)-dimensional identity matrix, we have
\begin{equation}\label{eq: eigenvalues}
    (Q_n-I_n) \chi_S =  -\frac{|S|}{n}\chi_{S  } .
\end{equation}  

Next, for each function \( f \colon \{ 0,1 \}^n \to \{ 0,1 \} \), we  define the matrices \( Q_{\partial f} \) and \( Q_f\) by
\begin{equation*}
    Q_{\partial f}(x,y) \coloneqq Q_n(x,y) \cdot \mathds{1}\bigl(f(x) \neq f(y)\bigr), \quad x,y \in \{ 0,1 \}^n
\end{equation*}
and 
\begin{equation*}
    Q_f \coloneqq Q_n - Q_{\partial f}.
\end{equation*} 
When \( p_n = 1/2 \), the function \( Q_{\partial f} \mathbf{1} \colon \{ 0,1 \}^n \to \mathbb{R} \) is exactly equal to the so-called \emph{sensitivity} of the function \( f \), sometimes denoted by \( h_f \) (see, e.g.,~\cite{eg} and~\cite{odonnell}). When \( p_n \neq 1/2 \), the function \( Q_{\partial f} \mathbf{1}\colon \{ 0,1 \}^n \to \mathbb{R}  \) can be thought of as a weighted analog of this function. 

Given \( f \colon \{ 0,1 \}^n \) and \( i \in [n]\), recall that we have defined \( I_i(f) \) to be the probability that \( f(X_0^{(n)}) \) changes when we re-randomize the \( i \)th bit of \( X_0^{(n)} \). In other words, we have
\begin{equation*} 
      I_i(f) =   \mathbb{E} \Bigl[   \pigl( \mathds{1}\bigl(X_0^{(n)}(i)=0\bigr) \cdot p_n + \mathds{1}\bigl(X_0^{(n)}(i)=1\bigr) \cdot (1-p_n) \pigr) \bigl(D_i f(X_0^{(n)})\bigr)^2 \Bigr] .
\end{equation*} 
Summing over all \( i \in [n]\), we obtain
\begin{equation}\label{eq: influence equation}
\begin{split}
      &I(f) =  \sum_{i \in [n]} \mathbb{E} \Bigl[   \pigl( \mathds{1}\bigl(X_0^{(n)}(i)=0\bigr) \cdot p_n + \mathds{1}\bigl(X_0^{(n)}(i)=1\bigr) \cdot (1-p_n) \pigr) \bigl(D_i f(X_0^{(n)})\bigr)^2 \Bigr] 
      = \pi_{n}^T Q_{\partial f} \mathbf{1}. 
\end{split}
\end{equation}  

 We will later be interested in how the  matrices \( Q_{\partial f} \) and  \( Q_f \)  acts on the functions in \( \{ \chi_S \}_{S \subseteq [n]} \). The first step in this direction is the following result.
 \begin{lemma}\label{lemma: edge matrix equality}
	Let \( f \colon \{ 0,1 \}^n \to \{ 0,1 \} \) and \( S \subseteq [n] \). Then
 	\begin{equation}\label{eq: exp II}
 		\pi_{n}^T Q_{\partial f} \chi_S
      	= 
     	\bigl\langle Q_{\partial f} \mathbf{1} , \chi_S \bigr\rangle.
	\end{equation}
\end{lemma}

 \begin{proof} 
 	Let \( D_{\pi_{n}} \) be the diagonal matrix with \( \diag D_{\pi_{n}} = \pi_n. \)
 	Then, since \( X^{(n)} \) is reversible, the matrix \( D_{\pi_n}^{1/2} Q_n D_{\pi_n}^{-1/2} \) is symmetric, and hence it immediately follows that \( D_{\pi_n}^{1/2} Q_{\partial f} D_{\pi_n}^{-1/2} \) is also symmetric. 
 	Using this observation, we obtain
 	\begin{equation*}
 		\begin{split}
 			&\langle  Q_{\partial f} \mathbb{1},\chi_S \rangle = (Q_{\partial f} \mathbb{1})^T D_{\pi_n} \chi_S
 			=
 			\mathbb{1}^T Q_{\partial f}^T D_{\pi_n} \chi_S
 			=
 			\mathbb{1}^T D_{\pi_n}^{1/2} (D_{\pi_n}^{1/2}Q_{\partial f} D_{\pi_n}^{-1/2})^T  D_{\pi_n}^{1/2} \chi_S
 			\\&\qquad
 			=
 			\mathbb{1}^T D_{\pi_n}^{1/2} (D_{\pi_n}^{1/2}Q_{\partial f} D_{\pi_n}^{-1/2})  D_{\pi_n}^{1/2} \chi_S
 			=
 			\mathbb{1}^T D_{\pi_n} Q_{\partial f} \chi_S
 			=
 			\pi_n^T Q_{\partial f} \chi_S
 		\end{split}
 	\end{equation*}
 	as desired. 
 \end{proof}

In the next lemma, we give an expression for \( \pi_n^T Q_{\partial f} \chi_S \) in terms of the Fourier coefficients of~\( f \).
\begin{lemma}\label{lemma: edge matrix composition increasing}
For any increasing function \( f \colon \{ 0,1 \}^n \to \{ 0,1 \} \) and \( S \subseteq [n] \), we have
\begin{equation} \label{eq: fourier expansion}
    \pi_n^T Q_{\partial f} \chi_S
    =  \frac{1}{n} \biggl[ (1-2p_n)  |S| \hat f(S  )  +  2\sqrt{p_n(1-p_n)}
       \sum_{i \in [n] \colon i\not \in S}  \hat f \bigl(S \cup \{ i \} \bigr)
     \biggr].
    \end{equation}
\end{lemma}

\begin{remark}
The proof of Lemma~\ref{lemma: edge matrix composition increasing} is the only part of the proof of our main result that directly requires that each function in the sequence \( (f_n)_{n \geq 1} \) is increasing. For comparison, one can show that the analogue of~\eqref{eq: fourier expansion} without this assumption is given by
\begin{equation*} 
\begin{split}
    &\pi_n^T Q_{\partial f} \chi_S
    =  \frac{\sqrt{p_n(1-p_n)}}{n} \!\!\!\! \sum_{\substack{T,T' \subseteq [n] \colon \\ T \Delta T' \subseteq S  \subseteq T \cup T'}} | T \cap T' |  \hat f \bigl(T  \bigr)\hat f \bigl(T'   \bigr) \Bigl( \frac{1-2p_n}{\sqrt{p_n(1-p_n)}} \Bigr)^{|S  \cap T \cap T'|} .
\end{split}
    \end{equation*}
    When \( p \neq 1/2 \), similar expressions appear naturally since for any \( f,g \colon \{ 0,1 \}^n \to \{ 0,1 \} \) and any \( S \subseteq [n] \), we have
\begin{equation*}
    \langle fg, \chi_S \rangle = \sum_{\substack{T,T' \subseteq [n]\colon \\T \Delta T' \subseteq S \subseteq T \cup T'}} \hat f(T) \hat g(T') \Bigl( \frac{1-2p_n}{\sqrt{p_n(1-p_n)}} \Bigr)^{|S \cap T \cap T'|}.
\end{equation*} 
The main reason we do not consider this general case is that it does not work as well with the inequalities we will apply later.
\end{remark}

\begin{proof}[Proof of Lemma~\ref{lemma: edge matrix composition increasing}]
    Let \( f \colon \{ 0,1 \}^n \to \{ 0,1 \} \) be increasing and let \( S \subseteq [n] \). By Lemma~\ref{lemma: edge matrix equality}, we have
    \begin{align*}
        \pi^T Q_{\partial f} \chi_S = \langle Q_{\partial f} \mathbf{1}, \chi_S \rangle.
    \end{align*}  
    Fix some \( x \in \{ 0,1 \}^n \). Since \( f \) is increasing  we have
    \begin{equation} \label{eq: first eq in proof}
        n  Q_{\partial f} \mathbf{1}(x)  
        =  \sum_{i\in [n]} \pigl(p_n \bigl(1-x(i) \bigr)  + (1-p_n) x(i) \pigr)      D_if(x)  .
    \end{equation}
    Note that for any \(i \in [n]\), we have  
\begin{equation*}
    \begin{split}
        &
        \bigl( 1 - x(i) \bigr) \cdot p_n + x(i) \cdot (1-p_n) 
         = 
        p_n(1-p_n)\biggl( 2 + \frac{1-2p_n}{\sqrt{p_n(1-p_n)}}  
        \cdot \frac{x(i)-p_n}{\sqrt{p_n(1-p_n)}}   \biggr)
        \\&\qquad= 
        p_n(1-p_n) \Bigl( 2 + \frac{1-2p_n}{\sqrt{p_n(1-p_n)}} \cdot \chi_{\{ i \}}(x)  \Bigr)
    \end{split}
\end{equation*}
    and that by Lemma~\ref{lemma: Partial i f}, we have
    \begin{equation*}
        D_if(x) 
        = \frac{1}{\sqrt{p_n(1-p_n)}} \sum_{T \subseteq [n] \colon i \not \in T} \hat f \bigl(T \cup \{ i \} \bigr) \chi_T(x).
    \end{equation*}
    Combining these observations with~\eqref{eq: first eq in proof}, we obtain
\begin{align*} 
    &n  Q_{\partial f} \mathbf{1}(x) \cdot  \chi_S (x)
    \\&\qquad =   \sum_{i \in [n]}   \pigl( 2\sqrt{p_n(1-p_n)}  + (1-2p_n) \cdot \chi_{\{ i \}}(x)\pigr)  \biggl[  \, \sum_{T \subseteq [n] \colon i \not \in T} \hat f \bigl(T \cup \{ i \} \bigr) \chi_T(x) \biggr]   \chi_S(x)
    \\&\qquad =   \sum_{i \in [n]}   \Bigl( 2\sqrt{p_n(1-p_n)}\sum_{T \subseteq [n] \colon i \not \in T} \hat f \bigl(T \cup \{ i \} \bigr) \chi_T(x)  + (1-2p_n) \cdot \sum_{T \subseteq [n] \colon i  \in T} \hat f \bigl(T   \bigr) \chi_{T }(x)\Bigr)    \chi_S(x).
    \end{align*} 
    Using~\eqref{eq: orthonormality}, we thus get 
	\begin{align*} 
    	&n  \sum_{x \in \{ 0,1 \}^n} \pi_n(x) Q_{\partial f} \mathbf{1}(x)  \chi_S (x)  
     	=   
      	\sum_{i \in [n]}   \pigl( 2\sqrt{p_n(1-p_n)} \hat f \bigl( S\cup \{ i \} \bigr)   \mathds{1}(i \not \in S)  + (1-2p_n)  \hat f \bigl(S   \bigr)  \mathds{1}(i \in S)\pigr)   
     	\\& \qquad=   
      	2\sqrt{p_n(1-p_n)}  \sum_{i  \in [n] \colon i \not \in  S}    \hat f \bigl( S\cup \{ i \} \bigr) 
      	+
       (1-2p_n)  |S| \hat f \bigl(S   \bigr) .
    \end{align*} 
    This concludes the proof.
\end{proof}

 In the next lemma we, given a Boolean function \( f \colon \{ 0,1 \}^n  \to \{0 , 1 \} \), express the moment generating function of \( C_f \) as a sum whose terms depend on the matrices \( Q_n \) and \( Q_{\partial f }\).
 
\begin{lemma}\label{lemma: generating C II}
    Let \( f \colon \{ 0,1 \}^n \to \{ 0,1 \} \). Then the moment generating function of \( C_f \) is given by
    \begin{equation*}
        \mathbb{E} [e^{sC_f}] = \sum_{k=0}^\infty \frac{n^k}{k!} {\pi_n}^T \Bigl( (Q_n-I_n) + (e^s-1)Q_{\partial f} \Bigr)^k \mathbf{1}, \quad s \in \mathbb{R}.
    \end{equation*}
\end{lemma}

\begin{proof}
    Let \( f \colon \{ 0,1 \}^n \to \{ 0,1 \} \) and let \( T_n \) be the (random) number of times in \( (0,1) \) when bits are resampled. By definition, \( T_n \) has a Poisson distribution with rate \( n \). This implies in particular that
    \begin{equation}\label{eq: mom gen fcn 1}
        \mathbb{E} [e^{sC_f}] = \mathbb{E} \Bigl[ \mathbb{E}\bigl[e^{sC_f}\mid T_n\bigl]\Bigl] = \sum_{k=0}^\infty \frac{e^{-n}n^k}{k!}\mathbb{E} \bigl[  e^{sC_f}\mid T_n=k\bigl].
    \end{equation}
    We now express \(  \mathbb{E}\bigl[  e^{sC_f}\mid T_n=k\bigl] \) in terms of the matrices \( Q_f \) and \( Q_{\partial f} \).  To this end, note first that each update of \( X^{(n)} \) corresponds to an entry  of \( Q_n.  \) Moreover, we have \( Q_n = Q_f + Q_{\partial f} \), and for each \( x,y \in \{ 0,1 \}^n \), we have either \( Q_n(x,y) = Q_f(x,y)\) and \( Q_{\partial f}(x,y)=0 \), or \( Q_n(x,y) = Q_{\partial f}(x,y)\) and \( Q_{f}(x,y)=0. \) Finally, note that if \( X^{(n)} \) jumps from \( x \) to \( y \) and \( Q_f(x,y) \neq 0 \), then we will get a contribution of \( 1 \) to \( e^{sC_f} \), while if \( Q_{\partial f}(x,y) \neq 0 \) then we get a contribution of \( e^s \) to \( e^{sC_f}. \) Using these observations, we find that
    \begin{equation}\label{eq: mom gen fcn 2}
        \begin{split}
            &\mathbb{E}\bigl[  e^{sC_f}\mid T_n=k\bigl] = \pi_n^T (Q_f + e^s Q_{\partial f})^k \mathbf{1}  
            = \pi_n^T \bigl((Q_f+Q_{\partial f}) + (e^s-1) Q_{\partial f}\bigr)^k \mathbf{1}
            \\&\qquad = \pi_n^T \bigl(Q_n + (e^s-1) Q_{\partial f}\bigr)^k \mathbf{1}.
        \end{split}
    \end{equation}
    Combining~\eqref{eq: mom gen fcn 1} and~\eqref{eq: mom gen fcn 2}, we obtain
    \begin{equation*}  
        \mathbb{E}[e^{sC_f}] 
        = {\pi_n}^T\sum_{k=0}^\infty \frac{e^{-n} n^k}{k!} \bigl(Q_n + (e^s-1) Q_{\partial f}\bigr)^k \mathbf{1}
        = \sum_{k=0}^\infty \frac{n^k}{k!} \pi_n^T \bigl( (Q_n-I_n) + (e^s-1)Q_{\partial f} \bigr)^k \mathbf{1}.
    \end{equation*}
    This concludes the proof.
\end{proof}

In the next lemma, we use the moment generating function of \( C_f \) given in Lemma~\ref{lemma: generating C II} to give expressions for the first and second moment of \( C_f \).

\begin{lemma}\label{lemma: moments II}
    Let \( f \colon \{ 0,1 \}^n \to \{ 0,1 \} \). Then
    \begin{equation}\label{eq: Cf expansion}
        \mathbb{E} [C_f] = n \cdot \pi_n^T    Q_{\partial f} \mathbf{1} 
    \end{equation}
    and
    \begin{equation}\label{eq: second moment I}
        \mathbb{E} \bigl[C_f^2\bigr] =  \mathbb{E}[C_f] + 2  \sum_{k=2}^\infty  \frac{n^k}{k!} {\pi_n}^T Q_{\partial f} (Q_n-I_n)^{k-2} Q_{\partial f}\mathbf{1}.
    \end{equation}
\end{lemma}
 
\begin{proof}
    By Lemma~\ref{lemma: generating C II}, for any \( s \in \mathbb{R} \), we have
    \begin{equation*}
        \mathbb{E} \bigl[e^{sC_f}\bigr] = \sum_{k=0}^\infty \frac{n^k}{k!} \pi_n^T \bigl( (Q_n-I_n) + (e^s-1)Q_{\partial f} \bigr)^k \mathbf{1}.
    \end{equation*}
    Differentiating with respect to \( s \) and using that \( \pi_n^T (Q_n-I_n) = (Q_n-I_n) \mathbf{1} = \mathbf{0} \), we get
    \begin{align*}
        &\mathbb{E} [C_f] 
        =
        \biggl[ \frac{d}{ds } e^{sC_f} \biggr]_{s = 0} 
        =
        \biggl[ \frac{d}{ds } \sum_{k=0}^\infty \frac{n^k}{k!} \pi_n^T \bigl( (Q_n-I_n) + (e^s-1)Q_{\partial f} \bigr)^k \mathbf{1} \biggr]_{s = 0} 
        =  n  \cdot  \pi_n^T Q_{\partial f} \mathbf{1}
    \end{align*}
    and
    \begin{align*}
        &\mathbb{E} \bigl[C_f^2 \bigr] 
        =
        \biggl[ \frac{d^2}{ds^2 } e^{sC_f} \biggr]_{s = 0} 
        = \biggl[ \frac{d^2}{ds^2}  \sum_{k=0}^\infty \frac{n^k}{k!} \pi_n^T \bigl( (Q_n-I_n) + (e^s-1)Q_{\partial f} \bigr)^k \mathbf{1} \biggr]_{s=0} 
        \\&\qquad=  n  \cdot  \pi_n^T Q_{\partial f}\mathbf{1} + 2 \sum_{k=2}^\infty \frac{n^k}{k!} \pi_n^T Q_{\partial f} (Q_n-I_n)^{k-2} Q_{\partial f} \mathbf{1}
    \end{align*} 
    which is the desired conclusion.
\end{proof}

 In the next lemma, we expand the terms in~\eqref{eq: second moment I} to get a simpler expression for the second moment of \( C_f \).
 
\begin{lemma}\label{lemma: not increasing nonsymmetric}
    If \( f \colon \{ 0,1 \}^n \to \{ 0,1 \}\), then
    \begin{align*}
        &\mathbb{E} [C_f^2]     
        = \mathbb{E}[C_f] +   \mathbb{E}[C_f]^2 + 2\sum_{S \subseteq [n] \colon S \neq \emptyset}  \frac{e^{-|S|}-(1-|S|)}{|S|^2} \bigl( n \cdot  \pi_n^T Q_{\partial f} \chi_S \bigr)^2. 
    \end{align*} 
\end{lemma}

\begin{proof} 
    Let \( f \colon \{ 0,1 \}^n \to \{ 0,1 \}.\) 
    Since
    \begin{align*} 
    & Q_{\partial f} \mathbf{1} (x) 
    =\sum_{S\subseteq [n]} \langle Q_{\partial f} \mathbf{1} ,\chi_S \rangle \chi_S
    \end{align*}  
    it follows from Lemma~\ref{lemma: edge matrix equality} that
    \begin{align*} 
    & Q_{\partial f} \mathbf{1} (x) 
    =\sum_{S\subseteq [n]} \bigl(\pi^T_n Q_{\partial f} \chi_S \bigr) \chi_S.
    \end{align*}  
     At the same time, for any \( S \subseteq [n] \), by~\eqref{eq: eigenvalues}, we have 
    \begin{align*} 
    &  \pi_n^T Q_{\partial f} (Q_n-I_n)^{k-2}  \chi_S
    =
    \pi_n^T Q_{\partial f} \bigl( (Q_n-I_n)^{k-2}  \chi_S \bigr)
    =
     \Bigl( \frac{-|S|}{n} \Bigr)^{k-2}  \pi_n^T Q_{\partial f}  \chi_S.
    \end{align*}  
    Combining these observations, we obtain
    \begin{align*} 
    &\pi_n^T Q_{\partial f} (Q_n-I_n)^{k-2} Q_{\partial f} \mathbf{1} 
    =
    \pi_n^T Q_{\partial f} (Q_n-I_n)^{k-2} \sum_{S\subseteq [n]} \bigl( \pi^T_n P_{\partial f} \chi_S\bigr) \chi_S
    \\&\qquad=
    \sum_{S\subseteq [n]} \bigl( \pi^T_n P_{\partial f} \mathbf{1}\bigr) \pi_n^T Q_{\partial f} (Q_n-I_n)^{k-2}  \chi_S
    =
    \sum_{S\subseteq [n]} \bigl( \pi^T_n P_{\partial f} \chi_S \bigr)^2 \Bigl( \frac{-|S|}{n} \Bigr)^{k-2}
    \\&\qquad=
    \frac{1}{n^2}\sum_{S \subseteq [n]} \Bigl( \frac{-|S|}{n}\Bigr)^{k-2}
    \bigl( n \cdot \pi^T Q_{\partial f} \chi_S \bigr)^2.
    \end{align*}   
    Using Lemma~\ref{lemma: moments II}, it immediately follows that 
\begin{align*}
	&\mathbb{E} [C_f^2] 
     =   
     \mathbb{E}[C_f] + 2\sum_{k=2}^\infty  \frac{n^k}{k!} \cdot  \frac{1}{n^2}\sum_{S \subseteq [n]} \Bigl( \frac{-|S|}{n}\Bigr)^{k-2}
    \bigl( n \cdot \pi^T Q_{\partial f} \chi_S \bigr)^2
     \\&\qquad =   
     \mathbb{E}[C_f] 
     + 
    \bigl( n \cdot \pi^T Q_{\partial f} \chi_\emptyset \bigr)^2 
    +
    2\sum_{S \subseteq [n] \colon S \neq \emptyset}  \bigl( n \cdot \pi^T Q_{\partial f} \chi_S \bigr)^2 \cdot \frac{1}{|S|^2}  \sum_{k=2}^\infty  \frac{n^k}{k!}   \Bigl( \frac{-|S|}{n}\Bigr)^{k}   
     \\&\qquad =   
      \mathbb{E}[C_f]  +  
     \bigl( n \cdot \pi^T Q_{\partial f} \chi_\emptyset    \bigr)^2
     +
     2\sum_{S \subseteq [n] \colon S \neq \emptyset} \bigl( n  \cdot \pi_n^T Q_{\partial f} \chi_S \bigr)^2 \cdot  \frac{e^{-|S|}-(1-|S|)}{|S|^2}.
\end{align*}
Recalling that by Lemma~\ref{lemma: moments II} we have 
\begin{equation*}
    \mathbb{E}[C_f] = n \cdot \pi^T Q_{\partial f} \mathbf{1}= n \cdot \pi^T Q_{\partial f} \chi_\emptyset,
\end{equation*}
the desired conclusion follow.
\end{proof} 

We now combine the lemmas in this section to give a  proof of  Proposition~\ref{proposition: increasing nonsymmetric}.
\begin{proof}[Proof of Proposition~\ref{proposition: increasing nonsymmetric}]
    By Lemma~\ref{lemma: edge matrix composition increasing}, for any \( S \subseteq [n] \) we have
    \begin{equation*} 
        \pi_n^T Q_{\partial f} \chi_S
        = \frac{1}{n} \Bigl[ (1-2p_n)  |S| \hat f(S  )  +  2\sqrt{p_n(1-p_n)}
        \sum_{i \not \in S}  \hat f \bigl(S \cup \{ i \} \bigr)
        \Bigr].
    \end{equation*} 
    Combining this with Lemma~\ref{lemma: not increasing nonsymmetric}, we obtain 
	\begin{align*}
		&\mathbb{E} [C_f^2]     
     	=   \mathbb{E}[C_f] +   \mathbb{E}[C_f]^2 + 2\sum_{S \subseteq [n] \colon S \neq \emptyset}  \frac{e^{-|S|}-(1-|S|)}{|S|^2} \bigl( n  \pi_n^T Q_{\partial f} \chi_S \bigr)^2
     	\\&\qquad=
     	\mathbb{E}[C_f] + \mathbb{E}[C_f]^2
        + 
     	\sum_{S \subseteq [n] \colon S \neq \emptyset } \frac{e^{-|S|}-(1-|S|)}{|S|^2} \\[-1.5ex]&\qquad\qquad\qquad\qquad\qquad\qquad\qquad\qquad\qquad \cdot \Bigl[  (1-2p_n)  |S| \hat f(S )  + 2\sqrt{p_n(1-p_n)} \sum_{i \in [n] \colon i \not \in S} \hat f \bigl(S \cup \{ i \} \bigr)
            \Bigr]^2 ,
	\end{align*}  
	which is the desired conclusion.
\end{proof}

\section{Proof of the main result}\label{sec: proof of main result}

In this section, we give a proof of Theorem~\ref{theorem: increasing gen p}.  The proof of this result will be divided into two lemmas, which we now state and prove.

\begin{lemma}\label{lemma: the last lemma}
    Let \( (f_n)_{n \in \mathcal{N}} \) be a sequence of increasing Boolean functions with \( \lim_{n \to \infty} \mathbb{E}[C_{f_n}] = \infty \), and assume that there is a constant \( C>0 \) such that for all \( n \in \mathcal{N} \) we have 
    \begin{equation}\label{eq: nontame inequality}
        p_n(1-p_n) n \Var(f_n) \leq C I(f_n)^2.
    \end{equation}
    Then \( (f_n)_{n\in \mathcal{N}} \) is not tame.
\end{lemma}

\begin{remark}\label{rem: rem with referee addition}
    By, e.g., Lemma~6.1 in~\cite{fk1996}, any increasing Boolean function \( f_n \colon \{ 0,1 \}^n \to \{0,1 \} \) satisfies \( I(f) \leq \sqrt{np_n} \). 
    Consequently, by Lemma~\ref{lemma: the last lemma}, any  sequence of increasing Boolean functions which is close to maximizing the total influence must be non-tame.
    
    If \( (f_n)_{n \in \mathcal{N}} \) is a sequence of regular and increasing Boolean functions with \( p_n = 1/2 \) and \( \sum I_i(f_n)^2 > c > 0 \) for all \( n \in \mathcal{N} \), then, as observed in the proof of Theorem~1.7 in~\cite{bks}, Theorem~1.1 in~\cite{Tal96} immediately implies that the functions \(f_n \) have uniformly positive correlations with the Majority function. See also Proposition~12.45 and~Theorem~12.51~in~\cite{gs}.
\end{remark}

\begin{proof}[Proof of Lemma~\ref{lemma: the last lemma}]
    By the Paley-Zygmund inequality, for any \( \theta \in (0,1) \) we have 
    \begin{equation*}
        P\bigl(C_{f_n} > \theta \mathbb{E}[C_{f_n}]\bigr) \geq (1 - \theta) \frac{\mathbb{E}[C_{f_n}]^2}{\mathbb{E}[C_{f_n}^2]}.
    \end{equation*}
    Since \( \lim_{n \to \infty} \mathbb{E}[C_{f_n}] = \infty \), it immediately follows that  if there is a constant \( C'>0 \) such that  
    \begin{equation}\label{eq: goal}
    	\mathbb{E}[C_{f_n}^2] \leq C'\mathbb{E}[C_{f_n}]^2 
    \end{equation}
    for all sufficiently large \( n \in \mathcal{N}, \) then \( (f_n)_{n \in \mathcal{N}} \)  is not tame.
    Consequently, if we can show that~\eqref{eq: goal} holds, then the desired conclusion will follow.
    To this end, fix some \( n \in \mathcal{N} \), and note that by~Proposition~\ref{proposition: increasing nonsymmetric}, since \( f_n \) is increasing, we have
    \begin{align*}
        \mathbb{E} [C_{f_n}^2]    
        &= 
        \mathbb{E}[C_{f_n}] + \mathbb{E}[C_{f_n}]^2
        + 
        2\sum_{S \subseteq [n] \colon S \neq \emptyset } \frac{e^{-S}-(1-|S|)}{|S|^2} \\&\qquad\qquad\qquad\cdot \Bigl[  (1-2p_n)  |S| \hat f_n(S  )  + 2\sqrt{p_n(1-p_n)} \sum_{i \not \in S} \hat f_n \bigl(S \cup \{ i \} \bigr)  
        \Bigr]^2 .
    \end{align*}
    Since \( e^{-|S|} \leq 1 \) for all \( S \subseteq [n] \), this implies that
    \begin{align*}
        \mathbb{E} [C_{f_n}^2]    
        & \leq
        \mathbb{E}[C_{f_n}] + \mathbb{E}[C_{f_n}]^2
        \\&\qquad+ 
        2\sum_{S \subseteq [n] \colon S \neq \emptyset } \frac{1}{|S|} \Bigl[  (1-2p_n)  |S| \hat f_n(S  )  + 2\sqrt{p_n(1-p_n)} \sum_{i \not \in S} \hat f_n \bigl(S \cup \{ i \} \bigr)  
        \Bigr]^2 .
    \end{align*}
    If we apply the inequality \( (a+b)^2 \leq 2(a^2 + b^2) \), we see that
    \begin{equation}\label{eq: inline eq 1}
    	\begin{split}
    		\mathbb{E} [C_{f_n}^2]    
        		&\leq
        		\mathbb{E}[C_{f_n}] +  \mathbb{E}[C_{f_n}]^2
        		+ 
        		4(1-2p_n)^2 \sum_{S \subseteq [n] \colon S \neq \emptyset }     |S| \hat f_n(S  )^2 
        		\\&\qquad+ 
        		16p_n(1-p_n)\sum_{S \subseteq [n] \colon S \neq \emptyset } \frac{1}{|S|}     \Bigl(\sum_{i\in [n] \colon i \not \in S} \hat f_n \bigl(S \cup \{ i \} \bigr)  \Bigr)^2. 
    	\end{split}	
    \end{equation}
    By the Cauchy-Schwarz inequality, we have
    \begin{equation}\label{eq: inline eq 2}
    	\begin{split}
    		&\sum_{S \subseteq [n] \colon S \neq \emptyset } \frac{1}{|S|}     \Bigl(\sum_{i\in [n] \colon i \not \in S} \hat f_n \bigl(S \cup \{ i \} \bigr)  \Bigr)^2
    		\leq 
    		\sum_{S \subseteq [n] \colon S \neq \emptyset } \frac{n-|S|}{|S|}     \sum_{i\in [n] \colon i \not \in S} \hat f_n \bigl(S \cup \{ i \} \bigr)^2
    		\\&\qquad=
    		\sum_{T \subseteq [n] \colon |T| \geq 2} \hat f(T)^2 \cdot \frac{|T|(n-(|T|-1))}{|T|-1}    
    		\leq
    		\sum_{T \subseteq [n] \colon |T| \geq 2} \hat f(T)^2 \cdot 2 n.
    	\end{split}
    \end{equation} 
    Now note that 
    \begin{equation*}
    	\sum_{S\subseteq [n]\colon S \neq \emptyset} |S|\hat f_n(S)^2 = I(f_n) = \mathbb{E}[C_n]
    \end{equation*}
    \begin{equation*}
    	\sum_{S\subseteq [n]\colon |S|\geq 2} \hat f_n(S)^2 \leq \sum_{S\subseteq [n]\colon S \neq \emptyset} \hat f_n(S)^2 = \Var(f_n).
    \end{equation*}
    Combining~\eqref{eq: inline eq 1} and~\eqref{eq: inline eq 2}, we thus obtain
    \begin{align*}
        \mathbb{E} [C_{f_n}^2]    
        &\leq
        \mathbb{E}[C_{f_n}] +  \mathbb{E}[C_{f_n}]^2
        + 
        4(1-2p_n)^2 \mathbb{E}[C_{f_n}]
        +
        32p_n(1-p_n) n \Var(f_n).
    \end{align*}
    Using~\eqref{eq: nontame inequality} and recalling that \( \mathbb{E}[C_{f_n}] = I(f_n) \), we thus obtain
    \begin{align*}
        \mathbb{E} [C_{f_n}^2]    
        &\leq
        \mathbb{E}[C_{f_n}] +  \mathbb{E}[C_{f_n}]^2
        + 
        4(1-2p_n)^2 \mathbb{E}[C_{f_n}]
        +
        32C \mathbb{E}[C_{f_n}]^2.
    \end{align*}
        Since \( \lim_{n \to \infty} \mathbb{E}[C_n] = \infty \), we have \( \mathbb{E}[C_n] \leq \mathbb{E}[C_n]^2 \) for all sufficiently large \( n \in \mathcal{N} \),  and hence for such \( n \),~\eqref{eq: goal} holds with, e.g., \( C' = 7+32C \). This concludes the proof.
 \end{proof}

It is well-known that, when \( p_n =1/2 \), the so-called \emph{Tribes function} (see, e.g., Section 4.2 in~\cite{odonnell}) is increasing and transitive and satisfies \( I(f_n) = C \log n \). Moreover, a sequence of Tribes functions is noise sensitive, and hence non-tame by Proposition~1.17 in~\cite{js2006}. In particular, this shows that the inequality in Lemma~\ref{lemma: the last lemma} does not hold for all non-tame sequences of increasing and transitive Boolean functions. The main idea in the proof of Theorem~\ref{theorem: increasing gen p} will instead be to show that~\eqref{eq: nontame inequality} holds for all sequences of regular Boolean functions which are not noise sensitive. This is the main motivation for the next lemma.

\begin{lemma}\label{lemma: the very last lemma}
    Let \( (f_n)_{n \in \mathcal{N}} \) be a non-degenerate sequence of regular and increasing Boolean functions. Assume further that \( p_n  \leq 1/2\) and \( \liminf_{n \to \infty} np_n = \infty\). Then either \( (f_n)_{n \in \mathcal{N}} \) is not tame, or 
    	\begin{equation*}
    		\lim_{n \to \infty} p_n^{-1}\sum_{i=1}^n I_i(f_n)^2 = 0.
    	\end{equation*}
\end{lemma}

\begin{proof}[Proof of Lemma~\ref{lemma: the very last lemma}]
    Assume that \( \limsup_{n \to \infty} p_n^{-1}\sum_{i \in [n]} I_i(f_n)^2 > 0 \). We need to show that in this case, \( (f_n)_{n \geq 1} \) is not tame. 
    To this end, note that when \( \limsup_{n \to \infty} p_n^{-1}\sum_{i \in [n]} I_i(f_n)^2 > 0 \), there is a constant \( D> 0 \) and an infinite increasing subsequence \( \mathcal{N}' = (n_1,n_2, \ldots)\) of \( \mathcal{N} \) such that \( \sum_{i \in [n]}  I_i(f_n)^2 > D p_n \) for all   \( n \in \mathcal{N}' .\)
    Since \( f_n \) is regular for all \( n \in \mathcal{N} \), and \( \mathcal{N}' \) is a subsequence of \( \mathcal{N} \), it follows that, for all \( n \in \mathcal{N}' \), we have
    \begin{equation*}
        I(f_n)^2 = \bigl(n I_1(f_n)\bigr)^2 = n \cdot   nI_1(f_n)^2
        = n\cdot  \sum_{i \in [n]} I_i(f_n)^2
        \geq n \cdot  D p_n = D \cdot p_n n.
    \end{equation*}
    Since \( \Var (f_n) (1-p_n) < 1 \) for all \( n \in \mathcal{N} \), it follows that 
    \begin{equation*}
    	p_n(1-p_n)n \Var(f_n) \leq D^{-1}I(f_n)^2.
    \end{equation*}
    On the other hand, since, by assumption, \( \lim_{n \to \infty} np_n = \infty \), \( p_n\leq 1/2 \), and \( (f_n)_{n \in \mathcal{N}} \) is non-degenerate, it follows from the previous equation to together with the observation that \( \mathbb{E}[C_f] = I(f_n) \) that \( \lim_{n \to \infty} \mathbb{E}[C_f] = \infty. \)
    Consequently, we can apply Lemma~\ref{lemma: the last lemma} to deduce that \( (f_n)_{n \in \mathcal{N}'} \) is not tame. Since \( \mathcal{N}' \) is a subsequence of \( \mathcal{N} \), it follows that \( (f_n)_{n \in \mathcal{N}} \) is not tame. This concludes the proof.
 \end{proof}
  
\begin{proof}[Proof of Theorem~\ref{theorem: increasing gen p}]
    Let \( (f_n)_{n \in \mathcal{N}} \) be a non-degenerate sequence of regular and increasing Boolean functions, and assume that \( 0 < \liminf_{n \to \infty} p_n \leq \limsup_{n \to \infty} < 1 \).
    Note that if \( (f_n)_{n \in \mathcal{N}} \) is noise sensitive then, by Proposition 1.17 in~\cite{js2006},  \( (f_n)_{n \in \mathcal{N}} \) is volatile and hence not tame. Consequently, we can assume that \( (f_n)_{n \in \mathcal{N}} \) is not noise sensitive. %
    Since \( 0<\liminf_{n \to \infty} p_n \leq \limsup_{n \to \infty} p_n < 1\), it follows from Theorem~7 in~\cite{gk2018} that \( \limsup_{n \to \infty} \sum_{i \in [n]} I_i(f_n)^2 > 0 \). The desired conclusion thus follows from applying Lemma~\ref{lemma: the very last lemma}.
\end{proof}

\begin{remark}
    The proof of Theorem~\ref{theorem: increasing gen p} does not work when \( \lim_{n \to \infty}p_n = 0 \), even though Theorem~7 in~\cite{gk2018}, as well as the similar Theorem~1.3 in~\cite{b}, holds also in this case. The reason for this is that when \( \lim_{n \to \infty} p_n = 0 \), the lower bound on \( p_n^{-1}\sum_{i \in [n]} I_i(f_n)^2 \) provided by these theorems are too weak to be used in conjunction with Lemma~\ref{lemma: the very last lemma}. 
    Theorem~I.5 in~\cite{lm2019} gives an alternative to Theorem~7 in~\cite{gk2018} for sequences of increasing Boolean functions when \( (p_n)_{n \geq 1} \) is given by \( p_n = n^{-(k-1)/k} \) for some even number \( k \geq 0 \). However,  this theorem has additional assumptions which, e.g., the counter-example given in~\cite{f2020} does not satisfy. More important for us however, it does not cover the range of \( (p_n)_{n \geq 1} \) where we have neither counter-examples nor a positive result.
\end{remark}

\section*{Acknowledgements}

The author would like to thank Gil Kalai and Jeffrey E. Steif for comments on the contents of this paper.
Also, the author is grateful to an anonymous referee for the many useful comments, including suggesting several improvements of the proofs in this paper, especially to the proof of~Lemma~\ref{lemma: edge matrix equality}, and also for pointing out the relationship to the Majority function, now mentioned in Remark~\ref{rem: rem with referee addition}.
Finally, the author is grateful to an anonymous referee on the companion paper~\cite{f2020}, for making several interesting comments of relevance for this paper.

 \end{document}